\documentclass[12pt]{amsart}
\usepackage{amssymb}
\usepackage{hyperref}

\newtheorem{thm}{Theorem}[section]
\newtheorem{lem}[thm]{Lemma}
\newtheorem{cor}[thm]{Corollary}
\newtheorem{prop}[thm]{Proposition}
\newtheorem{ex}[thm]{Example}

\newtheorem*{prob*}{Open problem}

\theoremstyle{definition}
\newtheorem{defi}[thm]{Definition}

\theoremstyle{remark}
\newtheorem{rem}[thm]{Remark}
\newtheorem*{rem*}{Remark}

\newcommand{\kringel}{\mathbin{\raise1pt\hbox{$\scriptstyle\circ$}}}
\newcommand{\pkt}{\mathbin{\raise0pt\hbox{$\scriptstyle\bullet$}}}

\newcommand{\C}{\mathbb{C}}

\newcommand{\N}{\mathbb{N}}

\newcommand{\R}{\mathbb{R}}

\newcommand{\ad}{{\rm ad}}

\newcommand{\Der}{{\rm Der}}

\newcommand{\Lf}{\mathfrak{f}}
\newcommand{\Lg}{\mathfrak{g}}

\newcommand{\Ll}{\mathfrak{l}}
\newcommand{\Ln}{\mathfrak{n}}
\newcommand{\Lm}{\mathfrak{m}}
\newcommand{\Lr}{\mathfrak{r}}
\newcommand{\Ls}{\mathfrak{s}}
\newcommand{\Lt}{\mathfrak{t}}

\newcommand{\al}{\alpha}
\newcommand{\be}{\beta}
\newcommand{\ga}{\gamma}

\newcommand{\la}{\lambda}

\newcommand{\Ga}{\Gamma}

\newcommand{\ov}{\overline}

\newcommand{\ra}{\rightarrow}

\renewcommand{\phi}{\varphi}

\newcommand{\Inn}{{\rm Inn}}
\newcommand{\AID}{{\rm AID}}
\newcommand{\CAID}{{\rm CAID}}

\begin{document}

% Ab hier duerfen Sie wieder.

\title[Almost Inner Derivations]{Almost inner derivations of Lie algebras}
%  Die Kurzfassung kommt oben ueber die Seiten, sie steht in eckigen Klammern
%  Auch Autorennamen koennen eine Kurzfassung haben

\author[D. Burde]{Dietrich Burde}
\author[K. Dekimpe]{Karel Dekimpe}
\author[B. Verbeke]{Bert Verbeke}
\address{Fakult\"at f\"ur Mathematik\\
Universit\"at Wien\\
  Oskar-Morgenstern-Platz 1\\
  1090 Wien \\
  Austria}
\email{dietrich.burde@univie.ac.at}
\address{Katholieke Universiteit Leuven Kulak\\
E. Sabbelaan 53 bus 7657\\
8500 Kortrijk\\
Belgium}
\email{karel.dekimpe@kuleuven-kulak.be}
\address{Katholieke Universiteit Leuven Kulak\\
E. Sabbelaan 53 bus 7657\\
8500 Kortrijk\\
Belgium}
\email{bert.verbeke@kuleuven-kulak.be}

\date{\today}


\begin{abstract}
We study almost inner derivations of Lie algebras, which were introduced by Gordon and Wilson in their
work on isospectral deformations of compact solvmanifolds. We compute all almost inner derivations for
low-dimensional Lie algebras, and introduce the concept of fixed basis vectors for proving that 
all almost inner derivations are inner for $2$-step nilpotent Lie algebras determined by graphs,
free $2$ and $3$-step nilpotent Lie algebras, free metabelian nilpotent Lie algebras on two generators, 
almost abelian Lie algebras and triangular Lie algebras. On the other hand we also exhibit families of 
nilpotent Lie algebras having an arbitrary large space of non-inner almost inner derivations. 
\end{abstract}

\maketitle
\section{Introduction}

Almost inner automorphisms of Lie groups and almost inner derivations of Lie algebras have been introduced 
by Gordon and Wilson \cite{GOW} in the study of isospectral deformations of compact solvmanifolds. 
A classical question going back to Hermann Weyl was whether or not isospectral manifolds are necessarily isometric. 
Milnor \cite{MIL} in $1964$ gave a negative answer by constructing two isospectral nonisometric flat tori in dimension $16$. 
Mark Kac in $1966$ gave the question the popular title ``Can One Hear the Shape of a Drum?'' The problem in two dimensions 
remained open until $1992$, when Gordon, Webb, and Wolpert constructed, based on the Sunada 
method, a pair of regions in the plane that have different shapes but identical eigenspectra. In $1984$ however, 
Gordon and Wilson wanted to construct not only finite families of isospectral nonisometric
manifolds, but rather continuous families. They constructed isospectral but nonisometric compact Riemannian manifolds 
of the form $G/\Ga$, with a simply connected exponential solvable Lie group $G$, and a discrete cocompact 
subgroup $\Ga$ of $G$. For this construction, almost inner automorphisms and almost inner derivations were crucial. \\
The concept of ``almost inner'' automorphisms and derivations, almost homomorphisms, or almost conjugate subgroups 
arises in many contexts in algebra, number theory and geometry. Another example here is the study on the relation between 
element-conjugacy and global conjugacy for algebraic groups by Larsen \cite{LAR}. This is, by a theorem of Sunada, 
closely related to the question of when a compact group can be the common covering space of a pair of non-isometric 
isospectral manifolds. \\
There are several other studies on related concepts, for example on local derivations \cite{AKU}, which are a generalization 
of almost inner derivations. \\[0.2cm]
The goal of our paper is to begin a systematic study of almost inner derivations of Lie algebras. 
Gordon and Wilson, and later others have given several examples of solvable and nilpotent Lie algebras and their almost 
inner derivations. However, the methods were ad hoc, and many examples were restricted to $2$-step nilpotent Lie 
algebras. \\
The paper is structured as follows. We first introduce almost inner and central almost inner derivations, and prove 
basic properties. Then we give examples of Lie algebras in dimension $5$ having non-inner almost inner derivations. 
In section $3$ we explain the concept of fixed basis vectors. This enables us to show, without computing all derivations, 
that for several kinds of Lie algebras all almost inner derivations are already inner. This includes $2$-step nilpotent 
Lie algebras determined by graphs, free $2$-and $3$-step nilpotent Lie algebras, free metabelian nilpotent Lie algebras 
on two generators, almost abelian Lie algebras and triangular Lie algebras. In section $7$ we prove that all metabelian 
filiform Lie algebras of dimension $n\ge 5$, except for the standard 
graded one, admit a non-inner almost inner derivation. In section $8$ we classify all complex Lie algebras of dimension $5$
admitting a non-inner almost inner derivation, and compute the space of almost inner derivations for all complex nilpotent
Lie algebras of dimension $6$. Finally, we construct infinite families of Lie algebras $\Lg$ having a space 
$\AID(\Lg)/\Inn(\Lg)$ of arbitrarily large dimension $n$, for any given $n\in \N$.

\section{Preliminaries}

Unless otherwise specified all Lie algebras we consider are over a general field $K$. 
The definition of almost inner derivations of Lie algebras in \cite{GOW} is as follows.

\begin{defi}\label{aid}
A derivation $D\in \Der(\Lg)$ of a Lie algebra $\Lg$ is said to be {\em almost inner}, if 
$D(x)\in [\Lg,x]$ for all $x\in \Lg$. The space of all almost inner derivations of $\Lg$ is denoted by $\AID(\Lg)$.
\end{defi}

A derivation is almost inner if and only if it coincides on each one-dimensional subspace with an inner 
derivation. In particular, the set of all inner derivations $\Inn (\Lg)$ is a subset of $\AID (\Lg)$.
Note that it is not enough in general to check the condition $D(x)\in [\Lg,x]$ only for basis vectors of $\Lg$.
We introduce a new subspace of $\AID(\Lg)$ as follows.

\begin{defi}\label{central}
An almost inner derivation $D\in \AID(\Lg)$ is called {\em central almost inner}
if there exists an $x\in \Lg$ such that $D-\ad (x)$ maps $\Lg$ to the center $Z(\Lg)$. We denote the space
of central almost inner derivations of $\Lg$ by  $\CAID(\Lg)$.
\end{defi}

The subspaces $\Inn(\Lg), \CAID(\Lg)$ and $\AID(\Lg)$ of $\Der(\Lg)$ become Lie subalgebras via the Lie bracket  $[D,D']=DD'-D'D$.

\begin{prop}
We have the following inclusions of Lie subalgebras
\[
\Inn(\Lg) \subseteq \CAID (\Lg) \subseteq \AID(\Lg)\subseteq \Der(\Lg).
\]
\end{prop}

\begin{proof}
Let $D,D'\in \AID(\Lg)$ and $x\in \Lg$. Then there exist $y,y'\in \Lg$ such that $D(x)=[y,x]$ and
$D'(x)=[y',x]$. Using the derivation rule and the Jacobi identity we obtain
\begin{align*}
[D,D'](x) & = (DD')(x)-(D'D)(x) \\
          & = D([y',x])-D'([y,x]) \\
          & = [D(y'),x]+[y',D(x)]-[D'(y),x]-[y,D'(x)]\\
          & = [D(y'),x]+[y',[y,x]]-[D'(y),x]-[y,[y',x]]\\
          & =  [D(y'),x]-[[y,y'],x]-[D'(y),x].
\end{align*}
Hence $[D,D'](x)=[D(y')-[y,y']-D'(y),x]\in [\Lg,x]$ for all $x\in \Lg$, so that $[D,D']\in \AID(\Lg)$. \\
Let $C,C'\in \CAID(\Lg)$.  Then there exist $y,y'\in \Lg$ such that $C-\ad(y)$ and $C'-\ad(y')$ map
$\Lg$ to $Z(\Lg)$. Using $[D,\ad(x)]=\ad (D(x))$ for $D\in \Der(\Lg)$ we obtain
\begin{align*}
[C-\ad(y),C'-\ad(y')] & = [C,C']-[C,\ad(y')]-[\ad(y),C']+[\ad(y),\ad(y')] \\
 & = [C,C']-\ad(C(y'))+\ad(C'(y))+\ad([y,y']),
\end{align*}
so that $[C,C']-\ad (C(y')-C'(y)-[y,y'])=[C-\ad(y),C'-\ad(y')]$ maps $\Lg$ to $Z(\Lg)$, and hence 
$[C,C']\in \CAID(\Lg)$.
\end{proof}

\begin{prop}
The subalgebra $\CAID (\Lg)$ is a Lie ideal in $\AID(\Lg)$, and $\Inn(\Lg)$ is a Lie ideal 
in all subalgebras of $\Der(\Lg)$ containing it.
\end{prop}

\begin{proof}
Let $C\in \CAID(\Lg)$ and $D\in \AID(\Lg)$. We need to show that $[D,C]\in \CAID(\Lg)$.
We already know that $[D,C]\in \AID(\Lg)$. Fix an element $x\in \Lg$ such that $C':=C-\ad(x)$ maps 
$\Lg$ to $Z(\Lg)$.  Define $D':=[D,C]-\ad (D(x))$. Then $D'=[D,C']$ because of $\ad(D(x))=[D,\ad(x)]$,
and $D'$ maps $\Lg$ to $Z(\Lg)$, since
\begin{align*}
D'(y) & = [D,C'](y) \\
      & = D(C'(y))-C'(D(y))
\end{align*}
for all $y\in \Lg$, because $C'$ maps $\Lg$ to $Z(\Lg)$ and $D$ maps $Z(\Lg)$ to $Z(\Lg)$. \\
Finally, $[D,\ad(x)]=\ad(D(x))$ is inner for all $D\in \Der(\Lg)$, so that $\Inn(\Lg)$ is an ideal
in all subalgebras of $\Der(\Lg)$ containing $\Inn(\Lg)$.
\end{proof}

\begin{rem}
We conjecture that $\AID(\Lg)$ is always a Lie ideal in $\Der(\Lg)$. However, this seems not
to be known, and there is no obvious algebraic argument for it.
\end{rem}

As a first example we will compute the almost inner derivations of the Heisenberg Lie algebra $\Ln_3(K)$,
given by the Lie brackets $[e_1,e_2]=e_3$. By this notation we will mean that $\Ln_3(K)$ is a 3-dimensional vector space over $K$ with basis vectors $e_1,\,e_2$ and $e_3$. The Lie brackets between basis vectors which are not specified are assumed to be zero, so $[e_1,e_3]=[e_2,e_3]=0$. 

\begin{ex}\label{2.6}
For $\Lg=\Ln_3(K)$ we have $\AID(\Lg)=\Inn(\Lg)$.
\end{ex}

Indeed, every derivation $D$ of $\Ln_3(K)$ is of the form $D(e_1)=\al_1e_1+\al_2e_2+\al_3e_3$,
$D(e_2)=\al_4e_1+\al_5e_2+\al_6e_3$, and $D(e_3)=(\al_1+\al_5)e_3$. Assume that $D$ is almost inner.
Then $D(e_1)\in [\Lg,e_1]=\langle e_3\rangle$, so that $\al_1=\al_2=0$. In the same way
$D(e_2)\in [\Lg,e_2]=\langle e_3\rangle$ implies that $\al_4=\al_5=0$, and $D(e_3)\in [\Lg,e_3]=0$ gives
$\al_1+\al_5=0$. It follows that $D\in \Inn(\Lg)$. \\[0.5cm]
The next examples show that there exist Lie algebras having more interesting almost inner derivations,
i.e., having non-inner almost inner derivations. Let $\Lg_{5,6}$ be the filiform nilpotent Lie algebra
with basis $\{e_1,\ldots ,e_5 \}$ and Lie brackets
\begin{align*}
[e_1,e_i] & = e_{i+1}, \; 2\le i\le 4, \\
[e_2,e_3] & = e_5.
\end{align*}
Let  $\Lg_{5,3}$ be the $3$-step nilpotent Lie algebra with basis $\{e_1,\ldots ,e_5 \}$ and Lie brackets
\[
[e_1,e_2] = e_4,\; [e_1,e_4]=e_5,\; [e_2,e_3]=e_5. \\
\]
Denote by $E_{ij}$ the matrix with entry $1$ at position $(i,j)$ and $0$ otherwise. As a linear map,
it maps $e_k$ to $0$ for $k\neq j$, and $e_j$ to $e_i$.

\begin{ex}\label{2.7}
For $\Lg=\Lg_{5,6}$ we have $\AID(\Lg)=\Inn(\Lg)\oplus \langle E_{5,2}\rangle$, and for 
$\Lg=\Lg_{5,3}$ we have $\AID(\Lg)=\Inn(\Lg)\oplus \langle E_{5,3}\rangle$.
\end{ex}

The proof follows again by a direct computation. The derivation $D=E_{5,2}$ for $\Lg_{5,6}$ is not inner,
but almost inner (in fact, central almost inner). The same holds for $D=E_{5,3}$ for $\Lg_{5,3}$.
We will see later also examples of Lie algebras $\Lg$ where the inclusion 
$\CAID(\Lg)\subseteq \AID(\Lg)$ is strict. \\[0.5cm]
The following table shows the result of a computation of almost inner derivations for all complex
nilpotent Lie algebras of dimension $5$. The classification of such Lie algebras is taken from \cite{MAG};
$c(\Lg)$ denotes the nilpotency class of $\Lg$, and $d(\Lg)$ the derived length. If the entry in the last column is non zero, it gives an example of an almost inner derivation which is not inner. 
\vspace*{0.5cm}
\begin{center}
\begin{tabular}{c|c|c|c|c|c|c|c}
Magnin & $c(\Lg)$ & $d(\Lg)$ & $\dim \Inn(\Lg)$  & $\dim \CAID(\Lg)$ &  $\dim \AID(\Lg)$ & $\dim \Der(\Lg)$ & $D $\\
\hline
$\Lg_{5,6}$ & $4$ & $2$ & $4$ & $5$ & $5$ & $8$  & $E_{5,2}$\\
\hline
$\Lg_{5,5}$ & $4$ & $2$ & $4$ & $4$ & $4$ & $9$  & $0$\\
\hline
$\Lg_{5,3}$ & $3$ & $2$ & $4$ & $5$ & $5$ & $10$  & $E_{5,3}$\\
\hline
$\Lg_{5,4}$ & $3$ & $2$ & $3$ & $3$ & $3$ & $10$  & $0$\\
\hline
$\Ln_4\oplus \C$ & $3$ & $2$ & $3$ & $3$ & $3$ & $11$  & $0$\\
\hline
$\Lg_{5,2}$ & $2$ & $2$ & $3$ & $3$ & $3$ & $13$  & $0$\\
\hline
$\Lg_{5,1}$ & $2$ & $2$ & $4$ & $4$ & $4$ & $15$  & $0$\\
\hline
$\Ln_3\oplus \C^2$ & $2$ & $2$ & $2$ & $2$ & $2$ & $16$  & $0$\\
\hline
$\C^5$ & $1$ & $1$ & $0$ & $0$ & $0$ & $25$  & $0$\\
\end{tabular}
\end{center}
\vspace*{0.5cm}
A further computation shows that $5$ is the minimal dimension where we have a complex Lie algebra admitting a non-inner
almost inner derivation:

\begin{prop}
Let $\Lg$ be a complex Lie algebra of dimension $n\le 4$. Then we have $\AID(\Lg)=\CAID(\Lg)=\Inn(\Lg)$.
\end{prop}

We will conclude this section with a few more easy facts on almost inner derivations.
Clearly $\Inn(\Lg)=\CAID(\Lg)=\AID(\Lg)=0$ for abelian Lie algebras.
Recall that a Lie algebra $\Lg$ is called {\em complete}, if $Z(\Lg)=0$ and $\Der(\Lg)=\Inn(\Lg)$.
Of course we have $\AID(\Lg)=\CAID(\Lg)=\Inn(\Lg)$ in this case. In particular
semisimple Lie algebras, and parabolic subalgebras of semisimple Lie algebras are complete.

\begin{prop}
Let $\Lg$ be a Lie algebra. Then the following statements hold.
\begin{itemize}
\item[$(1)$] Let $D\in \AID(\Lg)$. Then $D(\Lg)\subseteq [\Lg,\Lg]$, $D(Z(\Lg))=0$
and $D(I)\subseteq I$ for every ideal $I$ of $\Lg$. 
\item[$(2)$]  For $D\in \CAID(\Lg)$ there exists an $x \in \Lg$ such that  $D_{|[\Lg,\Lg]}=\ad (x)_{|[\Lg,\Lg]}$. 
\item[$(3)$] If $\Lg$ is $2$-step nilpotent, then $\CAID(\Lg)=\AID(\Lg)$. 
\item[$(4)$] If $Z(\Lg)=0$, then $\CAID(\Lg) = \Inn(\Lg)$. 
\item[$(5)$] If $\Lg$ is nilpotent, then  $\AID(\Lg)$ is nilpotent and all  $D\in \AID(\Lg)$ are nilpotent.
\item[$(6)$] We have $\AID(\Lg\oplus \Lg')=\AID(\Lg)\oplus \AID(\Lg')$ for the direct sum of two Lie algebras.
\end{itemize}
\end{prop}

\begin{proof}
By definition, an almost inner derivation maps $\Lg$ into $[\Lg,\Lg]$ and the center $Z(\Lg)$ to $0$.
Let $x\in I$. Then we have $D(x)\in [\Lg,x] \subseteq [\Lg,I]\subseteq I$, and $(1)$ follows. \\
Given $D\in \CAID(\Lg)$, there exists an $x\in \Lg$ such that $D'=D-\ad (x)$ satisfies $D'(\Lg)\subseteq Z(\Lg)$.
Since $D'$ is also a derivation we have
\[
D'([u,v]) = [D'(u), v ] + [ u, D'(v) ] = 0
\]
for all $u,v\in \Lg$. This shows $(2)$. If $\Lg$ is $2$-step nilpotent then $D(\Lg)\subseteq 
[\Lg,\Lg] \subseteq Z(\Lg)$ for any $D\in \AID(\Lg)$. Hence $\AID(\Lg)\subseteq \CAID(\Lg)$, and we have equality. 
This shows $(3)$. Suppose that $Z(\Lg)=0$ and $D\in \CAID(\Lg)$. Then there exists an $x\in \Lg$ such that 
$D(x)-\ad(x)=0$. Hence $D$ is inner. This shows $(4)$. 
%Let $D\in \AID(\Lg)$ and suppose that $\la$ is an eigenvalue
%of $D$. Then there is an $x\neq 0$ in $\Lg$ such that $\ad(y)(x)=D(x)=\la x$ for some $y\in \Lg$. 
%Since $\Lg$ is nilpotent, all $\ad(y)$ are nilpotent so that $\la=0$. 
% KAREL = theabove only works for \C.
Let $D\in \AID(\Lg)$  and $x\in \Lg$, then $D^k(x)\in [ \Lg,[\Lg,[\ldots,[\Lg,x]\ldots]]]$ ($k$  times $\Lg$). If $k$ is higher than the nilpotence class of $\Lg$, we have that $D^k(x)=0$, hence $D$ is nilpotent. By Engel's theorem
$\AID(\Lg)$ is nilpotent, and $(5)$ follows. For the last statement, let $D\in \AID(\Lg\oplus \Lg')$. Then the 
restrictions are again almost inner derivations, i.e., $D_{|\Lg}\in \AID(\Lg)$ and $D_{|\Lg'}\in \AID(\Lg')$.
It is easy to see that the map $D\mapsto D_{|\Lg}\oplus D_{|\Lg'}$ gives a one-to-one correspondence
between  $\AID(\Lg\oplus \Lg')$ and $\AID(\Lg)\oplus \AID(\Lg')$.
\end{proof}

\section{Fixed basis vectors}

For the computation of almost inner derivations of a given Lie algebra one does not always
need to know its derivation algebra explicitly. Instead one can use a concept, which we will call
{\em fixed basis vectors}. This is very useful, also for proving several results on almost inner derivations.
Unfortunately the definition is not particularly clear, although it is elementary. We will need to explain it with
some examples. 

For the rest of this section, $\Lg$ is an $n$-dimensional Lie algebra over a field $K$  and 
with {\em chosen basis} $\{e_1,\ldots ,e_n\}$. 
For $x\in \Lg$ we denote the centralizer of $x$ by $C_{\Lg}(x)=\{y\in \Lg\mid [x,y]=0 \}$.
Let $D$ be an almost inner derivation $D$ of $\Lg$. Then there exists a map $\phi_D\colon \Lg\ra \Lg$ such that
\[
D(x)=[x,\phi_D(x)]
\]
for all $x\in \Lg$. This map is not unique as we may change $\phi_D(x)$ to $\phi_D(x)+y$ for any $y\in C_{\Lg}(x)$.
It need also not be linear in general. If $x=\sum_{j=1}^n \al_je_j$, then we denote by $t_i(x)=\al_i$ the
$i$-th coordinate of $x$ with respect to the given basis.

\begin{defi}\label{3.1}
Let $D$ be an almost inner derivation of $\Lg$ determined by a map $\phi_D:\Lg\to \Lg$. We will say that
a basis vector $e_i$ is a {\em fixed vector for $D$ with fixed value $\al\in K$} if and only if for all
$j \in \{1,2,\ldots n\}:$
\[ 
\mbox{ if }e_j\not \in C_\Lg(e_i)\mbox{ then } t_i(\phi_D(e_j) )=\alpha.
\]
\end{defi}
Note that the $\al$ must be the same for all $j$ where this condition applies. As an example, consider
the Heisenberg Lie algebra $\Ln_3(K)$ with basis $\{e_1,e_2,e_3\}$ and Lie bracket $[e_1,e_2]=e_3$.

\begin{ex}
Let $\Lg=\Ln_3(K)$, and $D$ an almost inner derivation of $\Lg$ given by a map $\phi_D\colon \Lg\ra\Lg$.
Then every basis vector $e_i$ is fixed. 
\end{ex}

For $i=1$ we have $C_{\Lg}(e_1)=\langle e_1,e_3 \rangle$, and the condition just applies for $j=2$: since 
$e_2\not\in C_{\Lg}(e_1)$, we must have $t_1(\phi_D(e_2))=\al$. Certainly this is true, with the $\al$ given by
the map $\phi_D$. The same holds for $i=2$, where we have $C_{\Lg}(e_2)=\langle e_2,e_3 \rangle$. For $i=3$ we have
$C_{\Lg}(e_3)=\Lg$, so that the condition is vacuously true. \\[0.5cm]
The importance of finding fixed vectors comes from the following fact. If each basis vector for 
every almost inner derivations is fixed, then we have $\AID(\Lg)=\Inn(\Lg)$. We will prove this result in 
Corollary $\ref{allfixed}$.
Often we can show that every basis vector is fixed without knowing the structure 
of $\Der(\Lg)$. A trivial example is the following lemma.

\begin{lem}\label{3.3}
Let $\Lg$ be a Lie algebra with given basis $\{e_1,\ldots ,e_n\}$, such that for given $i$, the number of
basis vectors in $C_{\Lg}(x_i)$ is equal to $\dim(\Lg)$ or  $\dim(\Lg)-1$. Then the basis vector $e_i$ is fixed.
\end{lem}

\begin{proof}
In this case the condition for a fixed basis vector is vacuously true, or can be satisfied uniquely 
by the $\alpha$ given by the map $\phi_D$.
\end{proof}

We already saw this argument for $i=1,2,3$ in the example of $\Lg=\Ln_3(K)$ above. \\
We also want to present an example, where {\rm not} every basis vector is fixed.
For the Lie algebra $\Lg_{5,6}$ of Example $\ref{2.7}$ we will show that there is an almost inner
derivation $D$ determined by a map $\phi_D$ such that not every basis vector is fixed.

\begin{ex}
For $\Lg=\Lg_{5,6}$ and the almost inner derivation $D=E_{5,2}$ the basis vector $e_3$ is not fixed.
\end{ex}

We need to find a map $\phi_D$ representing $D$. Let $x=\al_1e_1+\ldots +\al_5e_5\in \Lg$. Define
$\phi_D$ as follows: 
\begin{itemize}
\item[$(1)$] If $\al_1\neq 0$, then $\phi_D(x)=\frac{\al_2}{\al_1}e_4$.
\item[$(2)$] If $\al_1=0$, then $\phi_D(x)=e_3$. 
\end{itemize}
It is easy to see that $D(x)=[x,\phi_D(x)]$ for all $x\in \Lg$. Definition $\ref{3.1}$ for
this $\phi_D$ and $i=3$ says: for all $j\in \{1,\ldots ,5 \}$, if $e_j\not\in C_{\Lg}(e_3)=\langle e_3,e_4,e_5\rangle$,
then $t_3(\phi_D(e_j))=\al$, each time for the same fixed $\al$. This applies for $j=1,2$, and we have
$\phi_D(e_1)=0$, $\phi_D(e_2)=e_3$, so that
\begin{align*}
t_3(\phi_D(e_2)) & = t_3(e_3)=1, \\
t_3(\phi_D(e_1)) & = t_3(0)=0. 
\end{align*}
So there is no fixed $\al$, and $e_3$ is not fixed.

\begin{lem}\label{3.5}
Let $D:\Lg \to \Lg$ be an almost inner derivation determined by a map $\phi_D:\Lg \to \Lg$.
If $e_i$ is a fixed basis vector with fixed value $\al$, then $D'=D+\ad (\alpha e_i)$ is an almost inner 
derivation which is determined by a map $\phi_{D'}:\Lg \to \Lg$ such that for all $j,k \in \{1,2,\ldots ,n\}:$
\begin{align*}
t_j(\phi_{D'}(e_k)) & = t_j(\phi_D(e_k)) \mbox{ for } i\neq j\\
t_i(\phi_{D'}(e_k)) & = 0. 
\end{align*}
\end{lem}

\begin{proof}
Clearly $D'$ is an almost inner derivation, and we have that
\begin{align*}
(D+\ad(\al e_i))(x)= [x,\phi_D(x)]+ [\al e_i,x] = [x, \phi_D(x) - \alpha e_i].
\end{align*}
So $D'$ is determined by the map $\tilde{\phi}_{D'}:\Lg \to \Lg: x \mapsto \phi_D(x) - \al e_i$. \\[0.2cm]
Now define the map
\[ 
\phi_{D'}: \Lg\to \Lg : x \mapsto \left\{ \begin{array}{ll}
\phi_D(x)-\al e_i & \mbox{ if }x\not\in \{e_1,e_2,\ldots,e_n\},\\
\phi_D(x)-t_i(\phi_D(x)) e_i & \mbox{ if }x\in \{e_1,e_2,\ldots,e_n\}.
\end{array}\right.
\]
We claim that $D'$ is also determined by this new map $\phi_{D'}$. Indeed, for all non basis vectors we have
$\phi_{D'}(x)=\tilde{\phi}_{D'}(x)$, so we only have to consider basis vectors. Let $e_j$ be a basis vector. Then there
are two possibilities:\\[0.2cm]
{\em Case 1:} $e_j\in C_\Lg(e_i)$. Then we have
\[
D'(e_j)  = D(e_j) = [e_j, \phi_D(e_j)] = [e_j, \phi_D(e_j)-t_i(\phi_D(x)) e_i] = [e_j, \phi_{D'}(e_j)]. 
\]
{\em Case 2:} $e_j\not \in C_\Lg(e_i)$. Then $t_i(\phi_D(e_j))=\al$, from which it follows that
$\tilde{\phi}_{D'}(e_j)=\phi_{D'}(e_j)$. \\[0.2cm]
Hence $D'$ is determined by $\phi_{D'}$. By definition of $\varphi_{D'}$ it is also easy to see that the requirements
$t_j(\phi_{D'}(e_k)) = t_j(\phi_D(e_k))$, for $j \neq i $, and $t_i(\phi_{D'}(e_k)) = 0 $ hold.
\end{proof}

As an immediate consequence we obtain the following result.

\begin{cor}\label{allfixed}
Let $D\in \AID(\Lg)$ be determined by a map $\phi_D$. If each basis vector is fixed, then $D\in \Inn(\Lg)$.
\end{cor}

\begin{proof}
Let $\al_i$ denote the fixed value of $e_i$. Then by iteratively applying Lemma $\ref{3.5}$, we find that
$D +\ad (\al_1 e_1) + \ad (\al_2 e_2) + \cdots + \ad (\al_n e_n) = D + \ad (v)$, with $v=\sum_{i=1}^n\al_i e_i$ 
is an almost inner derivation $D'$, determined by a map
$\phi_{D'}$ with $\phi_{D'}(e_i)=0$ for all $i\in \{1,2,\ldots,n\}$. This implies that $D'(e_i)=0$ for all basis vectors
$e_i$ and hence $D'=0$ or $D= -\ad (v)\in \Inn(\Lg)$.
\end{proof}

The next results are two technical lemmas, providing a way to find fixed basis vectors.
We will use the following notation: Let $i_1,i_2,\ldots, i_r\in \{1,2,\ldots ,n\}$ then
\[ \Lg_{i_1,i_2,\ldots,i_r}={\rm span} \{ e_i \;|\; i\not \in \{i_1,i_2,\ldots,i_r\}\}\]
denotes the vector space spanned by all basis vectors not in the set $\{e_{i_1},e_{i_2},\ldots,e_{i_r}\}$.

\begin{lem}\label{3.7}
Assume that $1\leq i,j,k,l,m\leq n$ and $l\neq m$. Moreover assume that there exist nonzero
scalars $\alpha,\beta \in K$  such that
\[\begin{array}{l}
{[e_j,e_i]- \alpha e_l \in \Lg_{l,m}}\\
{[e_k,e_i]-\beta e_m \in \Lg_{l,m}}\\
{[e_j,\Lg_i]\subseteq \Lg_{l,m}}\\
{[e_k,\Lg_i]\subseteq \Lg_{l,m}}.
\end{array}\]
Then, for any $D\in \AID(\Lg)$ determined by a map $\phi_D$, we have that $t_i(\phi_D(e_j))=t_i(\phi_D(e_k))$.
\end{lem}

\begin{proof}
Let $a = t_i(\phi_D(e_j))$, $b=t_i(\phi_D(e_k))$ and $c=t_i(\phi_D(e_j+e_k))$. Then there exist vectors $v,v',v''\in \Lg_i$ 
such that
\begin{align*}
\phi_D(e_j) & = a e_i + v, \\
\phi_D(e_k) & = b e_i+ v',\\
\phi_D(e_j+e_k) & = c e_i + v''.
\end{align*}
Using these notations we find that
\begin{equation}\label{expres1}
D(e_j+e_k)= [e_j+e_k, c e_i + v''] = c \al e_l + c \be e_m +w''
\end{equation}
for some $w''\in \Lg_{l,m}$, and on the other hand we have that
\begin{equation}\label{expres2}
D(e_j)+D(e_k)=[e_j,a e_i+v]+[e_k,b e_i+v']= a \al e_l + w + b \be e_m + w'  
\end{equation}
for some $w,w'\in \Lg_{l,m}$. Now, as $D$ is a linear map, the two expressions \eqref{expres1} and \eqref{expres2} 
must be equal, and so by comparing the $l$-th and $m$-th coordinate, we find that
\[ c \alpha = a \alpha,\; c\beta = b \beta.\]
As both $\alpha$ and $\beta$ are nonzero this implies that $a=b$ and hence
\[  t_i(\phi_D(e_j)) = t_i(\phi_D(e_k)).\]
\end{proof}

\begin{lem}\label{3.8}
Assume that $1\leq i,j,k,l\leq n$. Moreover assume that there exist nonzero
scalars $\alpha,\beta \in K$  such that
\[\begin{array}{l}
{[e_j,e_i]- \alpha e_l \in \Lg_{l}}\\
{[e_k,e_i]-\beta e_l \in \Lg_{l}}\\
{[e_j,\Lg_i]\subseteq \Lg_{l}}\\
{[e_k,\Lg_i]\subseteq \Lg_{l}}.
\end{array}\]
Then, for any $D\in \AID(\Lg)$ determined by a map $\phi_D$, we have that $t_i(\phi_D(e_j))=t_i(\varphi_D(e_k))$.
\end{lem}
\begin{proof}
Let $a=t_i(\phi_D(e_j))$, $b=t_i(\phi_D(e_k))$ and $c=t_i(\phi_D(\beta e_j -\alpha e_k))$.
Let $v,v',v''\in \Lg_i$ be such that
\begin{align*}
\phi_D(e_j) & = a e_i+ v, \\
\phi_D(e_k) & =b e_i+ v' \\
\phi_D(\beta e_j -\alpha e_k) & = c e_i + v''.
\end{align*}
Then we have that
\begin{equation}\label{expresa}
D(\beta e_j -\alpha e_k)=
[\beta e_j - \alpha e_k, c e_i + v'']= \beta c \alpha e_l -\alpha c \beta e_l +w'' 
\end{equation}
for some $w''\in \Lg_{l}$. On the other hand we have that
\begin{equation}\label{expresb}
\beta D(e_j) -\alpha D(e_k) = 
\beta [e_j, a e_i + v] - \alpha [e_k, b e_i + v']= \beta a \alpha e_l + w -\alpha b \beta e_l + w'
\end{equation}
for some $w,w'\in \Lg_l$. By comparing the $l$-th coordinate of \eqref{expresa} and \eqref{expresb} we find that
\[ \alpha \beta (a-b)=0 \Rightarrow a=b.\]
\end{proof}

\section{2-step nilpotent Lie algebras determined by graphs}

Let $G(V,E)$ be a finite simple graph with $V=\{x_1,x_2,\ldots,x_r\}$ its set of vertices and $E$ its set of edges. 
If there is an edge between vertex $x_i$ and $x_j$ with $i<j$, we denote this edge by the symbol $y_{i,j}$.
We let $X$ be the vector space over the field $K$ with basis the elements of $V$ and $Y$ be the vector space with basis the 
edges $y_{i,j}$. We define a two-step nilpotent Lie algebra $\Lg$ over $K$, where as a vector space $\Lg=X\oplus Y$ and 
where the brackets are given by
\begin{align*}  
[x_i,x_j] & =\begin{cases} y_{i,j}, \mbox{ if $y_{i,j} \in E$}\\
0, \mbox{ if there is no edge connecting $x_i$ with $x_j$}
\end{cases} \\[0.1cm]
[x_i, y_{j,k}]& = 0 \; \; \forall x_i \in V,\; \forall y_{j,k}\in E \\[0.1cm]
 [y_{i,j},y_{k,l}] & = 0 \; \; \forall y_{i,j},y_{k,l}\in E
\end{align*}

\begin{thm}
Let $\Lg$ be a 2-step nilpotent Lie algebra determined by a finite simple graph. Then $\AID(\Lg)=\Inn(\Lg)$.
\end{thm}
\begin{proof}
Let $s=\#E$ and choose an order $p_1, p_2, \ldots p_s$ for the edges. So any $p_t$ corresponds to a unique edge 
$y_{i,j}$. Now, we fix the basis $\{e_1,e_2,\ldots ,e_{r+s}\}$ of $\Lg$ given by 
\[ e_1=x_1,\ldots, e_r=x_r,e_{r+1}=p_1, e_{r+2}=p_2, \ldots, e_{r+s}=p_s.\]
Let $D\in \AID(\Lg)$ be determined by the map $\varphi_D$. We want to apply Corollary $\ref{allfixed}$, and hence 
we want to show that any basis vector is fixed for $D$. For $e_{r+1}, e_{r+2}, \ldots, e_{r+s}$ this is obvious,
since these vectors belong to $Z(\Lg)$. \\
Now, consider $e_i$ with $1\leq i \leq r$. If $e_i\in Z(\Lg)$, i.e.,  when $x_i$ is an isolated vertex, 
there is again nothing to show.
So assume that $e_i\not \in Z(\Lg)$. Then there is at least one $e_j\not \in C_\Lg(e_i)$ (with $1 \leq j \leq n$).
Hence $[e_j,e_i] = \pm e_l $ for some $l$ between $r+1$ and $r+s$. Let $\alpha= t_i(\varphi_D(e_j))$. Consider any other 
basis vector $e_k\not \in  C_\Lg(e_i)$. In order to show that $e_i$ is fixed, we must show that also 
$t_i(\varphi_D(e_k))=\alpha$. There exists an $m\in \{r+1,\ldots ,r+s\}$ with $[e_k,e_i]=\pm e_m$. As 
$\Lg$ is determined by a graph we have that $m\neq l$.

We are in the following situation
\[\begin{array}{l}
{[e_j,e_i] \pm  e_l =0}\\
{[e_k,e_i] \pm  e_m =0} \\
{[e_j,\Lg_i]\subseteq \Lg_{l,m}}\\
{[e_k,\Lg_i]\subseteq \Lg_{l,m}}.
\end{array}\]
This means that we can apply Lemma $\ref{3.7}$ and we find that $t_i(\varphi_D(e_k))=t_i(\varphi_D(e_j))=\alpha$.
Hence $e_i$ is indeed fixed for all $i$ and this finishes the proof.
\end{proof}

\begin{cor}\label{free2step}
Let $\Lf_{r,2}$ be the free 2-step nilpotent Lie algebra on $r$ generators, then 
\[ \AID(\Lf_{r,2})=\Inn(\Lf_{r,2}).\]
\end{cor}
\begin{proof}
This follows immediately from the fact that $\Lf_{r,2}$ is the 2-step nilpotent Lie algebra determined by the 
complete graph on $r$ vertices.
\end{proof}

\section{Free 3-step nilpotent Lie algebras}

Let $\Lf_{r,3}$ be the free 3-step nilpotent Lie algebra on $r$ generators $e_1,e_2,\ldots, e_r$. Having fixed 
these generators, we can find a Hall basis of $\Lf_{r,3}$, which is a basis of $\Lf_{r,3}$ as a vector space and 
which is explicitly given by the following collection of vectors:
\begin{align*}\label{Hall}
e_i &  \mbox{ for } 1\leq i \leq r \\[0.1cm]
y_{i,j} & = [e_i,e_j] \mbox{ for } 1\leq i < j \leq r \\[0.1cm]
z_{i,j,k} & = [e_i,y_{j,k}]\mbox{ for } 1 \leq j < k \leq r\mbox{ and } 1 \leq i \leq k.
\end{align*}
Note that if $i>k$ then 
\begin{align*}
[e_i,y_{j,k}]  & = [e_i,[e_j,e_k]] \\
 & =-[e_j,[e_k,e_i]] - [e_k,[e_i,e_j]] \\
 & = -z_{j,k,i}+z_{k,j,i}.
\end{align*}

\begin{lem} \label{zerointersect}
Let $x,y\in \Lf_{r,3}$. If $x- y\not\in [\Lf_{r,3},\Lf_{r,3}]$, then 
\[ [x,[\Lf_{r,3},\Lf_{r,3}]] \cap [y, [\Lf_{r,3},\Lf_{r,3}]] = 0.\]
\end{lem}
\begin{proof}
If either $x$ or $y$ belongs to $[\Lf_{r,3},\Lf_{r,3}]$ there is nothing to show. In case both do not belong 
to $[\Lf_{r,3},\Lf_{r,3}]$, the condition that $x- y\not\in [\Lf_{r,3},\Lf_{r,3}]$, actually means that we can choose 
a generating set $e_1=x,\; e_2=y,\; e_3, \ldots, e_r$ such that $\Lf_{r,3}$ is the free 3-step nilpotent Lie algebra
on that set of generators. Using the Hall basis introduced above, we see that 
\begin{align*} 
{[x, [\Lf_{r,3},\Lf_{r,3}]]} & = \langle z_{1,p,q}\;|\; 1 \leq p < q \leq r \rangle,  \\
{[y, [\Lf_{r,3},\Lf_{r,3}]]} & = \langle z_{2,p,q}\;|\; 1 \leq p < q \leq r \rangle.
\end{align*} 
Note that all of the vectors $z_{1,p,q}$ and $z_{2,p,q}$ belong to the Hall set mentioned above and that the set 
of basis vectors $z_{1,p,q}$ is disjoint of the set of basis vectors $z_{2,p,q}$. So we have that the 
subspaces spanned by those two sets have only the zero vector in common.
\end{proof}

\begin{thm} Let $f_{r,3}$ be the free $3$-step nilpotent Lie algebra on $r$ generators. Then
\[ \AID(\Lf_{r,3})= \Inn(\Lf_{r,3}).\]
\end{thm}

\begin{proof}
Let $D\in \AID(\Lf_{r,3})$. Note that $D$ induces an almost inner derivation $\bar{D}$ on 
$\Lf_{r,3}/Z(\Lf_{r,3})\cong \Lf_{r,2}$. By Corollary $\ref{free2step}$ we know that $\bar{D}$ is an inner derivation. 
Hence, by adjusting $D$ with an inner derivation, we may assume that $D(\Lf_{r,3})\subseteq Z(\Lf_{r,3})$. \\[0.2cm]
Let $e_1, e_2, \ldots , e_r$ be the generators of $\Lf_{r,3}$. Since we must have that $D(e_i)\in Z(\Lf_{r,3})$, there 
exist vectors $v_i\in [\Lf_{r,3},\Lf_{r,3}]$ such that 
\[ D(e_i)= [e_i, v_i]. \]
Analogously, there are also vectors $w_i \in  [\Lf_{r,3},\Lf_{r,3}]$, for $2\leq i \leq r$,  with 
\[ D(e_1 + e_i) = [e_1 + e_i, w_i].\]
By using the equation $D(e_1+e_i)=D(e_1) + D(e_i)$ we find that 
\[ [e_1,w_i] -[e_1,v_1] = [e_i,v_i] - [e_i,w_i].\]
Now, since the left hand side of the above expression belongs to $[e_1, [\Lf_{r,3},\Lf_{r,3}]]$ and the 
right hand side to $[e_i,[\Lf_{r,3},\Lf_{r,3}]]$,  it follows from Lemma $\ref{zerointersect}$
that both expressions are zero. Hence we have
\[ [e_1, w_i - v_1]= [e_i, w_i- v_i] =0 .\]
Since the only elements of $[\Lf_{r,3},\Lf_{r,3}]$ that commute with $e_1$, respectively with $e_i$, are 
those belonging to the center $Z(\Lf_{r,3})$, we find that 
\[ w_i - v_1\in Z(\Lf_{r,3}),\; w_i -v_i \in Z(\Lf_{r,3}) .\]
So $v_i-v_1\in Z(\Lf_{r,3})$. Therefore we can without any problem replace $v_i$ with $v_1$, and we 
find that $D(e_i)=[e_i,v_1]$. If we now consider the derivation $D'= D+\ad (v_1)$, we see that 
$D'(e_i)=0$. But then $D'$ is a derivation which is zero on the generators, and hence $D'$ is zero everywhere.
It follows that $D=-\ad (v_1)$, which was to be shown. 
\end{proof}

\section{Free metabelian nilpotent Lie algebras on two generators}

In this section we will show that all almost inner derivations are inner for free metabelian nilpotent Lie algebras
of class $c$ on $2$ generators. \\
Let $\Lf_2$ be the free Lie algebra on two generators, say $a$ and $b$. Let $\Lf_2^{(1)}=[\Lf_2,\Lf_2]$, 
$\Lf_2^{(i+1)}=[\Lf_2^{(i)},\Lf_2^{(i)}]$ for $i\ge 1$, and $\ga_1(\Lf_2)=\Lf_2$, $\ga_{i+1}(\Lf_2)=[\Lf_2,\ga_i(\Lf_2)]$  for $i\ge 1$. Then, the free $c$-step nilpotent and metabelian Lie algebra 
$\Lm_{2,c}$ is obtained as a quotient
\[ \Lm_{2,c}= \frac{\Lf_2}{\Lf_2^{(2)}+\gamma_{c+1}(\Lf_2)}.\]
So, it is the largest quotient of $\Lf_2$ which is both metabelian and $c$--step nilpotent.
Let us use $x_1$ and $x_2$ to denote the projection of $a$ and $b$ resp. in $\Lm_{2,c}$. 
We introduce the notation  $y^m_n$ for all $m\ge 2$, $ n\in \{1,\ldots, m-1\}$ by
\[  
y^{m}_{n}=[x_2,\underbrace{x_1,x_1,x_1,\ldots, 
x_1}_{\mbox{\small $m-n$ times }}, \underbrace{x_2,x_2,x_2,\ldots,x_2}_{\mbox{\small $n-1$ times }}],
\] where for all $z_1,z_2,\dots,z_n \in \Lg$, the iterated bracket $[[\dots[[z_1,z_2],z_3],\dots],z_n]$ is denoted with $[z_1,z_2,\dots,z_n]$.
So $y^m_n$ is an $m$-fold Lie bracket with $m-n$ appearances of $x_1$ and $n$ appearances of $x_2$.
It is well known that $x_1, x_2$ together with the elements $y^m_n$ ($1\leq n <m\leq c $) form a basis of $\Lm_{2,c}$ (E.g.~\cite[Section 4.7]{BAH}).
In fact, for any $i>1$  the projections of the elements $y^i_1,y^i_2,\ldots,y^i_{i-1}$ form a basis of 
$\gamma_{i}(\Lm_{2,c})/\gamma_{i+1}(\Lm_{2,c})$. So $\gamma_{i}(\Lm_{2,c})/\gamma_{i+1}(\Lm_{2,c})$ is 
$(i-1)$-dimensional (for $i\leq c$).

\begin{lem}
Let $z_1,z_2,\ldots, z_{n-2}\in \{x_1,x_2\}$ and $k=\#\{ i \in \{1,\ldots ,n-2\} \;|\; z_i=x_2\}+1$. Then
we have 
\[ [x_2,x_1,z_1,z_2,\ldots,z_{n-2}] = y^n_{k}.\]
\end{lem}

\begin{proof}
In a metabelian Lie algebra $\Lg$, it follows from the Jacobi identity that 
\[ \forall x,y \in \Lg, \forall c \in \gamma_2(\Lg):\; [c,x,y]=[c,y,x].\]
From this it follows that 
\[ [x_2,x_1,z_1,z_2,\ldots,z_{n-2}]= [x_2,x_1,z_{\sigma(1)},z_{\sigma(2)},\ldots, z_{\sigma(n-2)}]\]
for any permutation $\sigma$ on $n-2$ letters $1,2,\ldots,n-2$. Now, the result follows easily.
\end{proof}

The lemma easily implies the following identities.

\begin{cor} 
We have 
\[ [y^m_n,x_1]=y^{m+1}_n \mbox{ and }[y^m_n,x_2]= y^{m+1}_{n+1}.\]
\end{cor}

Now we can prove the main result of this section.

\begin{prop} 
Let $\Lm_{2,c}$ be the free $c$-step nilpotent and metabelian Lie algebra on 2 generators over an infinite field $K$.
Then $\AID(\Lm_{2,c})= \Inn(\Lm_{2,c})$.
\end{prop}

\begin{proof}
For $c=1$ we have that $\Lm_{2,c}$ is abelian and for $c=2$ we have that $\Lm_{2,c}$ is the Heisenberg 
Lie algebra. As these Lie algebras have no non-trivial almost inner derivations the proposition is 
valid in this situation.

For general $c\geq 3$, we proceed by induction and so we assume the proposition holds up to $c-1$.
Let $D$ be an almost inner derivation of $\Lm_{2,c}$. The space $I=\langle y^c_1, y^c_2,\ldots, y^c_{c-1} \rangle= 
\gamma_c(\Lm_{2,c})=Z(\Lm_{2,c})$ is an ideal of $\Lm_{2,c}$ and hence $D$ induces an almost inner 
derivation $\bar{D}$ on \[ \Lm_{2,c}/I\cong \Lm_{2,c-1}.\]
By the induction hypothesis, $\bar{D}$ is an inner derivation of $\Lm_{2,c-1}$. This means that we
can alter $D$ by an inner derivation of $\Lm_{2,c}$ and assume that 
\[D(\Lm_{2,c}) \subseteq I=\langle y^c_1, y^c_2,\ldots, y^c_{c-1} \rangle =\gamma_c(\Lm_{2,c}).\]
Moreover, by the fact that $D\in \AID(\Lm_{2,c})$ we must have that $D(x)\in [x, \Lm_{2,c}]$ and hence 
\[
D(x_1) \in \langle y^c_1, y^c_2, \ldots ,y^c_{c-2}\rangle \mbox{ and }
D(x_2) \in \langle y^c_2, y^c_3, \ldots ,y^c_{c-1}\rangle.
\]
So there are parameters $\alpha_1,\alpha_2, \ldots, \alpha_{c-2}, \beta_2, \beta_3, \ldots, \beta_{c-1} \in \C$ such that 
\[ 
D(x_1) = \alpha_1 y_1^c + \alpha_2 y_2^c + \cdots + \alpha_{c-2} y_{c-2}^c \mbox{ and }
D(x_2) = \beta_2 y_2^c + \beta_3 y_3^c + \cdots + \beta_{c-1} y_{c-1}^c.
\]

By changing $D$ to $D - \ad_{\alpha_1 y_1^{c-1} + \alpha_2 y_2^{c-1} + \cdots + \alpha_{c-2} y_{c-2}^{c-1}}$, 
we may assume that all parameters $\alpha_i=0$ ($1\leq i \leq c-2$) and we are in the situation with 
\[ 
D(x_1) = 0 \mbox{ and }
D(x_2) = \beta_2 y_2^c + \beta_3 y_3^c + \cdots + \beta_{c-1} y_{c-1}^c.
\]
Now let $\lambda\in K$. Then on the one hand we have that 
\begin{equation}\label{beta-cond}
D(\lambda x_1 +x_2) = \lambda D(x_1) +D(x_2) =  \beta_2 y_2^c + \beta_3 y_3^c + \cdots + \beta_{c-1} y_{c-1}^c.
\end{equation}
On the other hand, we also know that there exist an element $v_\lambda \in \Lm_{2,c}$ with 
\[  D(\lambda x_1 +x_2) = [ v_\lambda, \lambda x_1 +x_2].\]
Let 
\[ v_\lambda= a_1 x_1 + a_2 x_2 + \sum_{1\leq n <m\leq c} a_{m,n} y^m_n\]
then 
\begin{equation}\label{amn-cond}
 [ v_\lambda, \lambda x_1 +x_2] = 
 (a_2 \lambda -a_1) y_1^2 + 
 \sum_{1\leq n <m\leq c-1}\lambda  a_{m,n} y^{m+1}_n +  \sum_{1\leq n <m\leq c-1} a_{m,n} y^{m+1}_{n+1}
\end{equation}
Comparing the coefficients of the basis vectors $y^c_i$ of  \eqref{beta-cond} with \eqref{amn-cond} we get the following 
system of equations:
\[\left\{ \begin{array}{l}
\lambda a_{c-1,1} = 0\\
\lambda a_{c-1,2} + a_{c-1,1} = \beta_2 \\
\lambda a_{c-1,3} + a_{c-1,2} = \beta_3\\
\hspace*{1cm} \vdots\\
\lambda a_{c-1,c-2} + a_{c-1,c-3} = \beta_{c-2}\\
a_{c-1, c-2} = \beta_{c-1}.
\end{array} \right.
\]
This leads to \[\left\{ \begin{array}{l}
\lambda a_{c-1,1} = 0\\
\lambda^2 a_{c-1,2} + \lambda a_{c-1,1} = \lambda \beta_2 \\
\lambda^3  a_{c-1,3} + \lambda^2 a_{c-1,2} =\lambda^2  \beta_3\\
\hspace*{1cm} \vdots\\
\lambda^{c-2} a_{c-1,c-2} + \lambda^{c-3} a_{c-1,c-3} = \lambda^{c-3} \beta_{c-2}\\
\lambda^{c-2} a_{c-1, c-2} =\lambda^{c-2} \beta_{c-1}.
\end{array} \right.
\]
By taking the alternating sum of all these equations, we find that 
\[ \lambda \beta_2 - \lambda^2 \beta_3 + \cdots + (-1)^{c-2} \lambda^{c-3} \beta_{c-2} + (-1)^{c-1} \lambda^{c-2} \beta_{c-1}=0.\] 
Since the above equation has to hold for all possible $\lambda$ and $K$ is infinite, we must have that 
\[ \beta_2 = \beta_3 = \cdots = \beta_{c-1}=0.\]
It follows that $D(x_2)=0$. Together with the fact that $D(x_1)=0$ this implies that $D=0$, which means that 
the original $D$ we started with was an inner derivation.
\end{proof}

\section{Almost abelian Lie algebras and filiform nilpotent Lie algebras}

From now on we restrict ourselves to the case $K=\C$.
Almost abelian Lie algebras have no unique definition in the literature. A common convention is that
a Lie algebra $\Lg$ is {\em almost abelian} if it contains a $1$-codimensional abelian ideal. It is enough, however, to
require that $\Lg$ contains a $1$-codimensional abelian {\em subalgebra}, see \cite{BU39}.
Here we consider almost inner derivations of complex almost abelian Lie algebras. We may write
$\Lg=\C^n\rtimes \C$ with $\C=\langle e_{n+1}\rangle$, and a basis   
$e_1,e_2, \ldots, e_n$ of $\C^n$, such that with respect to this basis, $\ad (e_{n+1})_{\mid \C^n}$ is expressed in 
canonical Jordan form, i.e., 
\[ \ad (e_{n+1})_{\mid \C^n}=
\left( \begin{array}{ccccc}
B_1 & 0 & 0 & \cdots & 0\\
0 & B_2 & 0 & \cdots & 0 \\
0 & 0 & B_3 & \cdots & 0 \\
\vdots & \vdots & \vdots&  \ddots & \vdots \\
0 & 0 & 0 & \cdots & B_k 
\end{array}
\right)\]
where each $B_i$ is a block matrix of the form 
\[ B_i = 
\left( \begin{array}{cccccc}
\lambda_i & 1 & 0 & \cdots & 0& 0\\
0 & \lambda_i & 1 & \cdots & 0& 0 \\
0 & 0 & \lambda_i & \cdots & 0& 0 \\
\vdots & \vdots & \vdots&  \ddots &\vdots &  \vdots \\
0 & 0 & 0 & \cdots & \lambda_i &1 \\
0 & 0 & 0 & \cdots &0 &  \lambda_i 
\end{array}
\right)\]

We can apply the lemmas on fixed vectors to prove the following result.

\begin{prop}\label{7.1} 
Let $e_1,e_2,\ldots,e_n,e_{n+1}$ be the basis of $\Lg=\C^n \rtimes \C$ as described above. Then for any 
almost inner derivation  $D:\Lg\to \Lg$ determined by a map $\varphi_D$, any basis vector is fixed. It follows 
that $\AID(\Lg)=\Inn(\Lg)$.
\end{prop}

\begin{proof}
Let $i\in \{1,2,\ldots, n\}$, then all basis vectors $e_1,e_2,\ldots, e_n\in C_{\Lg}(e_i)$.
Hence $e_i$ is fixed by Lemma $\ref{3.3}$. So it suffices to show that $e_{n+1}$ is fixed. Therefore, we need to show 
that for any $e_j,e_k\not \in C_{\Lg}(e_{n+1})$ (with $1\leq j < k \leq n$) we have that 
\[ t_{n+1}(\varphi_D(e_j)) = t_{n+1}(\varphi_D(e_k)).\]
There are three different cases: \\[0.2cm]
{\em Case 1:} $e_j$ and $e_k$ are basis vectors for different Jordan blocks.  It follows that there 
exist $\lambda, \lambda'\in \C$ such that 
\begin{align*}
[e_{n+1}, e_j] & = \lambda e_j \mbox{ or }\lambda e_j + e_{j-1}, \\
[e_{n+1}, e_k] & = \lambda' e_k \mbox{ or } \lambda' e_k + e_{k-1}. 
\end{align*}
The two possibilities for each bracket are necessary for including the cases $\la=0$ or $\la'=0$.
In all of the situations above,  we can use Lemma $\ref{3.7}$, with $l=j$ or $j-1$ and $m=k$ or $k-1$,
to conclude that  $t_{n+1}(\varphi_D(e_j)) = t_{n+1}(\varphi_D(e_k))$. \\[0.2cm]
{\em Case 2:} $e_j$ and $e_k$ are basis vectors for the same Jordan block and $k-j \geq 2$. In this case we have exactly 
the same conclusion as in the previous case. \\[0.2cm]
{\em Case 3:} We have $k=j+1$ and $e_{j+1}$ and $e_{j}$ are basis vectors for the same Jordan block. In this case 
there is a $\lambda \in \C$ such that
\begin{align*}
[e_{n+1}, e_{j+1}] & = \lambda e_{j+1} + e_{j}, \\
[e_{n+1}, e_j]  & = \lambda e_j \mbox{ or } \lambda e_j + e_{j-1}. 
\end{align*}
If $\lambda\neq 0$, then Lemma $\ref{3.8}$, with $l=j$ allows us to conclude that 
$t_{n+1}(\varphi_D(e_{j+1})) = t_{n+1}(\varphi_D(e_j))$. On the other hand, if $\lambda=0$, then we must have that 
$[e_{n+1}, e_{j+1}]=e_j$ and $[e_{n+1},e_{j}]= e_{j-1}$, because otherwise $e_j\in C_\Lg(e_{n+1})$. In this case, we can again use 
Lemma $\ref{3.7}$, with $l=j-1$ and $m=j$, to conclude that  $t_{n+1}(\varphi_D(e_{j+1})) = t_{n+1}(\varphi_D(e_j))$.
\end{proof}

Denote by $\Lf_n$ the standard graded filiform nilpotent Lie algebra of dimension $n$, defined by the
Lie brackets $[e_1,e_i]=e_{i+1}$ for $i=2,\ldots ,n-1$ in the basis $(e_1,\ldots ,e_n)$. Clearly we have
$\Lf_n\cong \C^{n-1}\rtimes \C$ with $\C=\langle e_1\rangle$ over $\C$. Hence we obtain the following result
as a corollary of Proposition $\ref{7.1}$.

\begin{prop}\label{7.2}
The filiform nilpotent Lie algebra $\Lf_n$ satisfies $\AID(\Lf_n)=\CAID(\Lf_n)=\Inn(\Lf_n)$.
\end{prop}

We already have seen in Example $\ref{2.7}$, that filiform nilpotent Lie algebras 
can have more interesting almost inner derivations than just inner ones. The algebra $\Lg_{5,6}$ in this
example is metabelian filiform. It turns out that this example generalizes to all metabelian filiform
Lie algebras of dimension $n\ge 5$. It has been shown in  \cite{BRA} that every metabelian filiform Lie algebra $\Lg$
of dimension $n\ge 3$  has an {\em adapted basis} $(e_1,\ldots,e_n)$ such that
\begin{align*}
[e_1,e_i] & = e_{i+1}, \; 1\le i\le n-1 \\
[e_2,e_k] & = \al_{2,5}e_{2+k}+\cdots + \al_{2,n-k+3}e_n,\; 3\le k\le n-2 \\
[e_i,e_k] & =0,\; i,k\ge 3,
\end{align*}
with structure constants $\{\al_{2,k}\mid 5\le k\le n\}$. Clearly $\Lg\cong\Lf_n$ if and only if all structure 
constants are zero.

\begin{lem}
Let $\Lg$ be a complex metabelian filiform Lie algebra of dimension $n\ge 3$ and let $D\in \AID(\Lg)$. 
Then there exists a $v \in \Lg$ and a $\lambda \in \C$ such that
\[ D -\ad_v = \lambda E_{n,2}.\]
\end{lem}

\begin{proof}
We proceed by induction on the dimension $n$. If $n<5$, then $\Lg$ is a standard filiform Lie algebra and all almost inner 
derivations are inner by Proposition $\ref{7.2}$. So the result holds, with $\lambda = 0$.
So assume that $n\geq 5$ and that the lemma is valid for metabelian filiform Lie algebras of smaller dimensions. 
Let $D\in \AID(\Lg)$. Then $D$ induces an almost inner derivation $\bar{D}$ on $\Lg/\langle e_n \rangle$. By induction, we may 
assume, after changing $D$ up to an inner derivation, that we have  
$\bar{D}= \mu E_{n-1,2}$ for some $\mu \in \C$. This implies that $D(e_1)= a e_n$ for some $a\in \C$. Now, replace 
$D$, with $D'= D+ \ad_{a\,e_{n-1}}$. Then we have
\begin{align*} 
D'(e_1) & =D(e_1) + [a e_{n-1}, e_1]=0, \\
D'(e_i) & = D(e_i) + [a e_{n-1}, e_i] = D(e_i) \mbox{ for } i\geq 2.
\end{align*}
In particular, we have that 
\[
 D'(e_2) = D(e_2)= \mu e_{n-1} + \lambda e_{n} \mbox{ for some }\mu,\lambda\in \C.
\] 
From this it follows that 
\begin{align*}
D'(e_3) & = D'[e_1,e_2]= [D'(e_1), e_2]+ [e_1, D'(e_2)]= \mu e_n, \\
D'(e_4) & = D'[e_1,e_3]=  [D'(e_1), e_3]+ [e_1, D'(e_3)]=0 ,
\end{align*}
and analogously $D'(e_i)=0$ for $i \geq 5$. To finish the proof, we have to show that $\mu=0$. So
assume that $\mu\neq 0$. \\
Since we have $D'(e_3) = \mu e_n$ and $D'\in \AID(\Lg)$, there must exist an element $\sum_{i=1}^n a_i e_i\in \Lg$ 
with $[\sum_{i=1}^n a_i e_i , e_3] = \mu e_n$. This leads to the equation 
\[ a_1 e_4 + a_2[e_2,e_3] = \mu e_n,\]
which expands to 
\[ a_1 e_4 + a_2 (\alpha_{2,5} e_5+ \alpha_{2,6} e_6 + \cdots + \alpha_{2,n} e_n) = \mu e_n. \]
As we assume that $\mu\neq 0$, this implies  
\[ \alpha_{2,5}=\alpha_{2,6} = \cdots = \alpha_{2,n-1}=0.\]
As a conclusion thus far, we have found that when $\mu \neq 0$, then  the basis vectors $e_i$ satisfy
\begin{align*}
[e_1,e_i] & = e_{i+1}, \; 2\le i\le n-1 \\
[e_2,e_3] & = \al_{2,n} e_n\\
[e_2, e_i] & = 0, \; 4 \le i \le n\\
[e_i,e_j] & = 0,\; i,j\geq 3.
\end{align*}
There must also exist an element $\sum_{i=1}^n b_i e_i\in \Lg$ with $D(e_2) = [  \sum_{i=1}^n b_i e_i, e_2]$.
This leads to the equation 
\[ \mu e_{n-1} + \lambda e_n = b_1 e_3 -  b_3 \al_{2,n} e_n.\] 
Since we are assuming that $\mu \neq 0$, this equation does not have a solution, which is a contradiction. 
Hence indeed $\mu=0$, and therefore $D'= \lambda E_{2,n}$, which was to be shown.
\end{proof}

The lemma now easily implies the following result.

\begin{prop}
Let $\Lg$ be a metabelian filiform Lie algebra of dimension $n\ge 5$, which is different from $\Lf_n$. Then 
\[
\AID(\Lg)=\CAID(\Lg)=\Inn(\Lg)\oplus \langle E_{n,2}\rangle.
\]
\end{prop}

\begin{proof}
We only have to show that $D=E_{n,2}$ is an almost inner derivation. If $x=\sum_{i=1}^n\be_ie_i$ with
$\be_1\neq 0$, then $D(x)=\be_2e_n=[x,\frac{\be_2}{\be_1}e_{n-1}]$. Otherwise $\be_1=0$. Since $\Lg$ is not the
standard graded algebra $\Lf_n$, there exists a minimal index $i$ with $5\le i\le n$ such that $\al_{2,i}\neq 0$.
Then, for $k=n-i+3\ge 3$ we have $D(x)=\be_2e_n=[x,\frac{1}{\al_{2,i}}e_k]$. Hence $D(x)\in [\Lg,x]$ for all $x\in \Lg$.
\end{proof}

\begin{rem}
There are also filiform nilpotent Lie algebras $\Lg$ with $\dim (\AID(\Lg)/\Inn(\Lg))\ge 2$ for $\dim (\Lg)\ge 7$;
of course with $d(\Lg)\ge 3$. The following table shows the dimensions of the derivations spaces for all complex filiform nilpotent
Lie algebras of dimension $7$, with Magnin's notation \cite{MAG}. For $\Lg_{7,1.1(i_{\la})}$ we have $\la\neq 0,1$.
The two cases $\la=0$ and $\la=1$ are listed separately. The last column, when non zero, gives examples of almost inner 
derivations, which together with the inner derivations generate $\AID(\Lg)$.
\vspace*{0.5cm}
\begin{center}
\begin{tabular}{c|c|c|c|c|c|c|c}
Magnin & $c(\Lg)$ & $d(\Lg)$ & $\dim \Inn(\Lg)$  & $\dim \CAID(\Lg)$ &  $\dim \AID(\Lg)$ & $\dim \Der(\Lg)$ & $D $\\
\hline
$\Lg_{7,0.1}$ & $6$ & $3$ & $6$ & $7$ & $8$ & $10$ & $E_{6,2}+E_{7,3},E_{7,2}$\\
\hline
$\Lg_{7,0.2}$ & $6$ & $2$ & $6$ & $7$ & $7$ & $10$ & $E_{7,2}$ \\
\hline
$\Lg_{7,0.3}$ & $6$ & $2$ & $6$ & $7$ & $7$ & $11$  & $E_{7,2}$\\ 
\hline
$\Lg_{7,1.1(i_{\la})}$ & $6$ & $3$ & $6$ & $7$ & $8$ & $10$  & $E_{6,2}+E_{7,3},E_{7,2}$\\ 
\hline
$\Lg_{7,1.1(i_0)}$ & $6$ & $3$ & $6$ & $6$ & $6$ & $10$  & $0$\\ 
\hline
$\Lg_{7,1.1(i_1)}$ & $6$ & $2$ & $6$ & $7$ & $7$ & $11$  & $E_{7,2}$\\ 
\hline
$\Lg_{7,1.1(ii)}$ & $6$ & $3$ & $6$ & $7$ & $7$ & $11$  & $E_{7,2}$\\ 
\hline
$\Lg_{7,1.4}$ & $6$ & $2$ & $6$ & $7$ & $7$ & $12$  & $E_{7,2}$\\ 
\hline
$\Lg_{7,1.6}$ & $6$ & $2$ & $6$ & $7$ & $7$ & $12$  & $E_{7,2}$\\ 
\hline
$\Lg_{7,2.3}$ & $6$ & $2$ & $6$ & $6$ & $6$ & $13$  & $0$\\ 
\hline
\end{tabular}
\end{center}
\end{rem}

\section{Low-dimensional Lie algebras}

Complex Lie algebras of dimension $n\le 4$ do not have non-inner almost inner derivations.
This is different in dimension $5$. In order to determine the space of almost inner derivations
we will not use a full classification of all $5$-dimensional Lie algebras, but rather a description 
of the moduli space given in \cite{FIP}. Here the authors give a  natural stratification by orbifolds, 
in terms of $24$ families of Lie algebras, with up to $4$ parameters.
This is much better than a full classification for us, because the determination of almost inner derivations 
is much more efficient for the stratification, the list of the full classification being much too long.
We already know from the table after Example $\ref{2.7}$ that every complex nilpotent Lie algebra of dimension
$5$ having a non-inner almost inner derivation is isomorphic to $\Lg_{5,3}$ or $\Lg_{5,6}$. \\ 
The most interesting family of solvable, non-nilpotent Lie algebras in this context is the family 
$d_{12}(p:q:r)$ with $p=0$ from \cite{FIP}.

\begin{defi}
The family of complex $5$-dimensional Lie algebras $A(q,r)=d_{12}(0:q:r)$ is defined by the Lie brackets
\begin{align*}
[e_1,e_5] & = e_2, \\
[e_2,e_5] & = (q+r)e_2, \\
[e_3,e_4] & = e_2, \\
[e_3,e_5] & = e_1+qe_3, \\
[e_4,e_5] & = e_3+re_4. \\
\end{align*}
\end{defi}

It is straightforward to compute the almost inner derivations of this family.

\begin{lem}\label{8.2}
We have
\begin{align*}
\dim \Der (A(q,r)) & = \begin{cases} 7, & \text{if } (q,r)\neq(0,0) \\ 
 8, & \text{otherwise} \end{cases} \\
\dim \Inn (A(q,r)) & = \hspace{0.35cm} 4, \hspace{0.4cm} \text{for all } q,r\in \C\\
\dim \AID (A(q,r)) & = \begin{cases} 4, & \text{if } qr\neq 0, q+r\neq 0 \\ 
5, & \text{otherwise} \end{cases} 
\end{align*}
\end{lem}

We can determine the Lie algebras $A(q,r)$ with $\dim \AID (A(q,r))=5$ up to isomorphism.

\begin{lem}
Every Lie algebra  $A(q,r)$ satisfying $qr=0$ or $q+r=0$ is either isomorphic to $A(1,0)$, to
$A(1,-1)$ or to $A(0,0)\cong \Lg_{5,6}$.  
\end{lem}

\begin{proof}
Note that $A(q,r)\cong A(r,q)$, see \cite{FIP}. It is easy to see that $A(0,0)$ is filiform nilpotent and
isomorphic to $\Lg_{5,6}$. So we may assume that $(q,r)\neq (0,0)$. 
Suppose first that $qr=0$. Then we may assume $q\neq 0$ and $r=0$, and there is an Lie algebra
isomorphism $\phi\colon A(q,0)\ra A(1,0)$ given by
\[
\begin{pmatrix} 
q^2 & 0 & 0 & 0 & 0 \\
1-q^2 & q & 0 & 0 & 0 \\
0 & 0 & q & 0 & 0 \\
0 & 0 & 0 & 1 & \frac{1-q^2}{q} \\
0 & 0 & 0 & 0 & q
\end{pmatrix}.
\]
Secondly, let $q+r=0$ and $q\neq 0$. Then there is an Lie algebra isomorphism $\phi\colon A(q,-q)\ra A(1,-1)$ 
given by
\[
\begin{pmatrix} 
1 & 0 &  \frac{q^2-1}{q}  &  \frac{q^2-1}{q^2}  & 0 \\
0 & q &  \frac{q^2-1}{q} & 0 & 0 \\
0 & 0 & q & 0 & 0 \\
0 & 0 & 0 & 1 & 0 \\
0 & 0 & 0 & 0 & q
\end{pmatrix}.
\]
Finally, $A(q,r)$ is unimodular if and only if $q+r=0$. Hence $A(1,-1)$ is unimodular, but
$A(1,0)$ is not. So they cannot be isomorphic. Both $A(1,-1)$ and $A(1,0)$ are solvable and non-nilpotent,
whereas $A(0,0)$ is nilpotent.
\end{proof}

\begin{prop}
Every complex Lie algebra of dimension $5$ having a non-inner almost inner derivation is isomorphic
to one of the following Lie algebras:
\[
\Lg_{5,3},\; \Lg_{5,6},\; A(1,0),\; A(1,-1).
\]
\end{prop}

\begin{proof}
We use Table $3$ of \cite{FIP} listing the $24$ families of Lie algebras. For each family, or type,
we compute the spaces $\Der(\Lg)$, $\AID(\Lg)$ and $\Inn(\Lg)$ for all possible parameters.
The types $d_1,d_2,d_3,d_4,d_7,d_8,d_{10},d_{11}, d_{16},d_{17},d_{18}, d_{19}, d_{24}$ have no parameters, so that 
the computation is easy. Also, the nilpotent algebras are easy, because they correspond to choosing all parameters
equal to zero. Moreover we do know the result already for nilpotent algebras.
Note that there is an error in the Lie brackets of $d_2$ in table $3$ of \cite{FIP}, where $\psi^{12}_1$ has to be 
removed; and also in the definition on page $429$.
The hardest cases are the ones with $3$ or $4$ parameters, namely the families $d_5(p:q:r)$,
$d_{12}(p:q:r)$, $d_{20}(p:q:r:s)$ and $d_{21}(p:q:r)$. A long, but straightforward computation shows that, for
non-nilpotent algebras, the only family with non-inner almost inner derivations is $d_{12}(p:q:r)$, where
we need $p=0$. More precisely we see that only for the algebras $A(q,r)=d_{12}(0:q:r)$ with
$q=0$, or $r=0$ or $q+r=0$ this is the case. We obtain (see Lemma $\ref{8.2}$)
\[
\AID(A(q,r)) =
\begin{cases} 
\Inn (A(q,r))\oplus\langle {E_{2,4}}\rangle & \text{for } q=0,r\neq 0 \\
\Inn (A(q,r))\oplus\langle {E_{2,4}}+q{E_{3,5}}\rangle & \text{for } r=0,q\neq 0 \\
\Inn (A(q,r))\oplus\langle r^2{E_{2,1}}-{E_{2,4}}+r {E_{3,5}}+r^2{E_{4,5}}\rangle & \text{for } q+r= 0 \\
\end{cases}
\]
\end{proof}

\begin{rem}
The Lie algebra $A(1,-1)$ arises in a different context, namely in the classification of
Lie algebras admitting a Sasakian structure, see \cite{AFV}. A Sasakian structure on a Riemannian
manifold is the analogue in odd dimensions of a K\"ahler structure. Indeed, a Riemannian manifold
$M$ of odd dimension admits a compatible Sasakian structure if and only if the Riemannian cone
$M\times \R^+$ is K\"ahler. Left-invariant Sasakian structures on Lie groups can be classified by
Sasakian structures on its Lie algebras. In the classification of $5$-dimensional Sasakin Lie algebras 
in Theorem $10$ of \cite{AFV}, the Sasakian Lie algebra $\Lg_3$ is isomorphic to $A(1,-1)$ over the complex
numbers.
\end{rem}

\begin{rem}\label{8.6}
In dimension $6$ we have computed the almost inner derivations only for nilpotent Lie algebras, using
the classification given in \cite{MAG}. The result is given in the following table:
\vspace*{0.5cm}
\begin{center}
\begin{tabular}{c|c|c|c|c|c|c|c}
Magnin & $c(\Lg)$ & $d(\Lg)$ & $\dim \Inn(\Lg)$  & $\dim \CAID(\Lg)$ &  $\dim \AID(\Lg)$ & $\dim \Der(\Lg)$ & ${D} $\\
\hline
$\Lg_{6,20}$ & $5$ & $3$ & $5$ & $5$ & $6$ & $8$ & $E_{5,2}$\\
\hline
$\Lg_{6,18}$ & $5$ & $3$ & $5$ & $5$ & $5$ & $9$ & $0$ \\
\hline
$\Lg_{6,19}$ & $5$ & $2$ & $5$ & $6$ & $6$ & $9$  & $E_{6,2}$\\ 
\hline
$\Lg_{6,17}$ & $5$ & $2$ & $5$ & $6$ & $6$ & $10$  & $E_{6,2}$\\ 
\hline
$\Lg_{6,15}$ & $4$ & $2$ & $5$ & $5$ & $5$ & $10$  & $0$\\ 
\hline
$\Lg_{6,13}$ & $4$ & $2$ & $5$ & $6$ & $6$ & $10$  & $E_{6,3}$\\ 
\hline
$\Lg_{6,16}$ & $5$ & $2$ & $5$ & $5$ & $5$ & $11$  & $0$\\ 
\hline
$\Lg_{6,14}$ & $4$ & $2$ & $4$ & $4$ & $4$ & $11$  & $0$\\ 
\hline
$\Lg_{6,9}$ & $3$ & $2$ & $5$ & $5$ & $5$ & $11$  & $0$\\ 
\hline
$\Lg_{6,12}$ & $4$ & $2$ & $5$ & $5$ & $5$ & $11$  & $0$\\
\hline
$\Lg_{5,6}\oplus\C$ & $4$ & $2$ & $4$ & $5$ & $5$ & $12$  & $E_{5,2}$\\ 
\hline
$\Lg_{6,5}$ & $3$ & $2$ & $4$ & $4$ & $4$ & $12$  & $0$\\ 
\hline
$\Lg_{6,10}$ & $3$ & $2$ & $5$ & $5$ & $5$ & $12$  & $0$\\ 
\hline
$\Lg_{6,11}$ & $4$ & $2$ & $5$ & $6$ & $6$ & $12$  & $E_{6,3}$\\ 
\hline
$\Lg_{5,5}\oplus\C$ & $4$ & $2$ & $4$ & $4$ & $4$ & $13$  & $0$\\ 
\hline
$\Lg_{6,8}$ & $3$ & $2$ & $4$ & $6$ & $6$ & $13$  & $E_{5,3}, E_{6,2}$\\ 
\hline
$\Lg_{6,4}$ & $3$ & $2$ & $4$ & $4$ & $4$ & $13$  & $0$\\
\hline
$\Lg_{6,7}$ & $3$ & $2$ & $4$ & $6$ & $6$ & $14$  & $E_{6,2}, E_{6,3}$\\ 
\hline
$\Lg_{6,2}$ & $3$ & $2$ & $5$ & $5$ & $5$ & $14$  & $0$\\
\hline
$\Lg_{6,6}$ & $3$ & $2$ & $4$ & $4$ & $4$ & $15$  & $0$\\ 
\hline
$\Lg_{5,4}\oplus \C$ & $3$ & $2$ & $3$ & $3$ & $3$ & $15$  & $0$\\ 
\hline
$\Lg_{5,3}\oplus \C$ & $3$ & $2$ & $4$ & $5$ & $5$ & $15$  & $E_{5,3}$\\ 
\hline
$\Ln_3\oplus \Ln_3$ & $2$ & $2$ & $4$ & $4$ & $4$ & $16$  & $0$\\
\hline
$\Ln_4\oplus \C^2$ & $3$ & $2$ & $3$ & $3$ & $3$ & $17$  & $0$\\ 
\hline
$\Lg_{6,1}$ & $2$ & $2$ & $4$ & $6$ & $6$ & $17$  & $E_{6,3}, E_{6,4}$\\ 
\hline
$\Lg_{6,3}$ & $2$ & $2$ & $3$ & $3$ & $3$ & $18$  & $0$\\ 
\hline
$\Lg_{5,2}\oplus \C$ & $2$ & $2$ & $3$ & $3$ & $3$ & $19$  & $0$\\ 
\hline
$\Lg_{5,1}\oplus \C$ & $2$ & $2$ & $4$ & $4$ & $4$ & $21$  & $0$\\ 
\hline
$\Ln_3\oplus \C^3$ & $2$ & $2$ & $2$ & $2$ & $2$ & $24$  & $0$\\ 
\hline
$\C^6$ & $1$ & $1$ & $0$ & $0$ & $0$ & $36$  & $0$\\ 
\end{tabular}
\end{center}
\end{rem}

\section{Triangular Lie algebras}

In this section we consider the Lie algebra $\Lt_n(K)$, resp.\ $\Ln_n(K)$, of all upper-triangular, resp.\ 
strictly upper-triangular, $n \times n$ matrices over a general field $K$ again.

Let $e_{i,j}$ denote the  $n\times n$ matrix with 0's everywhere, except a 1 on the $(i,j)$-th spot.
Recall that 
\begin{equation}\label{bracket eij}
[e_{i,j}, e_{k,l}] = \delta_{j,k} e_{i,l} - \delta_{l,i} e_{k,j}.
\end{equation}

\begin{prop} For any $n\geq 2$ we have that 
\[ \AID(\Ln_n(K)) = \Inn(\Ln_n(K)).\]
\end{prop}

\begin{proof}
The Lie algebra $\Ln_n(K)$ has a basis consisting of the matrices $e_{i,j}$ where $1 \leq i < j \leq n$. From
\eqref{bracket eij} it follows that 
\[[ \Ln_n(K),e_{i,j}]= \langle e_{i,j+1}, e_{i,j+2}, \ldots , e_{i,n}, e_{1,j},e_{2,j}, \ldots ,e_{i-1,j}\rangle.\]
We proceed by induction on $n$. 
For $n=2$ the proposition is trivially true since $\Ln_2(K)\cong K$ is abelian. So we assume that $n>2$ and that the 
result holds for smaller values of $n$. 

Let $D \in \AID(\Ln_n(K))$. Note that $I=\langle e_{1,2}, e_{1,3},\ldots,e_{1,n}\rangle$ is an ideal of $\Ln_n(K)$. This implies that 
$D(I)\subseteq I$. It follows that $D$ induces a derivation $\bar{D}$ of $\Ln_n(K)/I\cong \Ln_{n-1}(K)$.
Of course $\bar{D}\in \AID(\Ln_{n-1}(K))$ and by induction, we can conclude that $\bar{D}$ is an inner derivation.
Let $x\in \Ln_n(K)$ be an element such that $\bar{D}=\ad (\bar{x})$ where $\bar{x}$ denotes the projection of $x$ 
in $\Ln_{n-1}(K)\cong \Ln_n(K)/I$. \\
By replacing $D$ by $D-\ad (x)$ we may assume that $D$ is an almost inner derivation of $\Ln_n(K)$ with 
$D(\Ln_n(K))\subseteq I$. It follows that there exist elements $\beta_3, \beta_4, \ldots ,\beta_n\in K$ such that 
\[ D(e_{2,3})=\beta_3 e_{1,3},\; D(e_{3,4})=\beta_4 e_{1,4},\; D(e_{4,5})=\beta_5 e_{1,5},\;\ldots,
D(e_{n-1,n})=\beta_n e_{1,n}.\]
Let $a=\beta_3 e_{1,2} + \beta_4 e_{1,3} + \cdots + \beta_n e_{1,n-1}$, then for all $i$ with 
$2 \leq i  \leq n-1$ we have that 
\begin{eqnarray*}
\ad (a)(e_{i,i+1}) & = & [ \beta_3 e_{1,2} + \beta_4 e_{1,3} + \cdots + \beta_n e_{1,n-1},e_{i,i+1}] \\
& = & [ \beta_{i+1} e_{1,i}, e_{i,i+1} ] \\
& = & \beta_{i+1} e_{1,i+1} \\
& = & D( e_{i,i+1}).
\end{eqnarray*} 
So, by replacing $D$ with $D-\ad (a)$, we may assume that 
\[ D(e_{2,3})=D(e_{3,4})=\cdots = D(e_{n-1,n})=0.\]
Note that $\ad (a)(I)\subseteq I$, so that also after modifying $D$, we still have that $D(\Ln_n(K))\subseteq I$. \\
There also exist $\alpha_3,\alpha_4,\ldots,\alpha_n\in K$ with 
\[ D(e_{1,2})= \alpha_3 e_{1,3} + \alpha_4 e_{1,4} + \cdots + \alpha_n e_{1,n} .\]
For $3\leq i < n $ we have $[e_{1,2},e_{i,i+1}]=0$, so that 
\begin{align*}
0 & = D[e_{1,2}, e_{i,i+1}]\\
& = [D(e_{1,2}), e_{i,i+1}] + [e_{1,2}, D(e_{i,i+1})]\\
& = [\alpha_3 e_{1,3} + \alpha_4 e_{1,4} + \cdots + \alpha_n e_{1,n}, e_{i,i+1}]+ 0\\
& = \alpha_i e_{1,i+1}.
\end{align*}
It follows that $\alpha_i =0 $ for all $n>i \geq 3$, so that
\[ D(e_{1,2})= \alpha_n e_{1,n} = \ad (-\alpha_n e_{2,n})(e_{1,2}).\]
Note that for $i\geq 2$ we have $\ad (-\alpha_n e_{2,n})(e_{i,i+1})=0$. So by finally replacing $D$ with 
$D+\ad (\alpha_n e_{2,n})$ we find that 
\[ D(e_{i,i+1})=0,\; \forall 1\leq i <n.\]
But this implies that $D=0$, so that the original $D$ is in $\Inn(\Ln_n(K))$, which was to be shown.
\end{proof}

By exactly the same technique one can also prove the following result:

\begin{prop} For any $n\geq 2$ we have  
\[ \AID(\Lt_n(K)) = \Inn(\Lt_n(K)).\]
\end{prop}

\section{Nilpotent Lie algebras with arbitrary large $\AID(\Lg)/\Inn(\Lg)$}

In the previous sections we had many negative results concerning the existence of non-inner almost inner derivations. 
We want to show now that it is also possible to construct infinite families of Lie algebras $\Lg$ having 
a space $\AID(\Lg)/\Inn(\Lg)$ of arbitrarily large dimension $n$, for any given $n\in \N$. \\[0.5cm]
Consider the following family of $2$-step nilpotent Lie algebras $\Lg_n$ over a general field $K$ of dimension $4n+2$, 
with basis 
\[ t_1,\; t_2,\;x_{1,i},\; x_{2,i},\; y_{1,i},\; y_{2,i} \;\; (1 \leq i \leq n)\]
and non-zero Lie brackets  
\[ [t_1, x_{1,i}] = y_{1,i},\; [t_1,x_{2,i}] = y_{2,i},\; [t_2, x_{2,i}]= y_{1,i}\;\; (1 \leq i \leq n).\]
So we have $\Lg_n=K^{4n}\rtimes K^2$, where $K^{4n}$ is the subspace spanned by the $x_{p,i}$'s and the $y_{p,i}$'s and 
$K^2$ is spanned by $t_1$ and $t_2$.

\begin{prop} For every $n\geq 2$ we have 
\[ \dim (\AID(\Lg_n)/\Inn(\Lg_n)) = n.\]
\end{prop}

\begin{proof}
Any element $x$ of $\Lg_n$ can be written uniquely in the form 
\[
x=\alpha_1 t_1 + \alpha_2 t_2 + v 
\]
where $v \in K^{4n}=\langle x_{p,i},\;y_{p,i},\; 1\leq p \leq 2, \;1\leq i \leq n\rangle$.  
Using this notation we define for any $i=1,2,\ldots, n$ a map 
\[ \varphi_{D_i}:\Lg_n\to \Lg_n: x=\alpha_1 t_1 + \alpha_2 t_2 + v \mapsto 
\left\{ \begin{array}{lcl} 
0 &\mbox{if}& \alpha_1 =0 \\
-\frac{\alpha_2}{\alpha_1} x_{1,i} + x_{2,i}& \mbox{if}& \alpha_1 \neq 0.
\end{array} \right.\]
Now let 
\[D_i:\Lg_n\to \Lg_n: x \mapsto D_i(x):=[ x, \varphi_{D_i}(x)]. \]
For $\alpha_1\neq 0$ we have that 
\begin{eqnarray*}
D_i(\alpha_1 t_1 + \alpha_2 t_2 + v) & = & [ \alpha_1 t_1 + \alpha_2 t_2 + v, -\frac{\alpha_2}{\alpha_1} x_{1,i} + x_{2,i}]\\
& = & -\alpha_2 y_{1,i} + \alpha_1 y_{2,i} + \alpha_2 y_{1,i}\\
& = & \alpha_1 y_{2,i} 
\end{eqnarray*}
Also for $\alpha_1=0$, we have that $D_i(\alpha_1 t_1 + \alpha_2 t_2 + v) = 0 = \alpha_1 y_{2,i}$. 
Hence $D_i:\Lg_n \to \Lg_n$ is a linear map, having its image in the center of $\Lg_n$, and so $D_i$ is a derivation. 
By construction, $D_i\in \AID(\Lg_n)$. We claim that the set of all $D_i+ \Inn(\Lg_n)$ ($1\leq i \leq n$) forms 
a basis of $\AID(\Lg_n)/\Inn(\Lg_n)$. \\[0.5cm]
We will first show that this is a linearly independent set. Assume that 
$\sum_{i=1}^n \beta_i D_i\in \Inn(\Lg_n)$, then 
\[ 
\sum_{i=1}^n \beta_i D_i =\ad (\alpha_1 t_1 + \alpha_2 t_2 + v)
\]
for some $\alpha_1,\alpha_2 \in K,\; v \in K^{4n}$. As $ \sum_{i=1}^n \beta_i D_i(x_{1,1}) = 0$, it follows that 
$0 = [\alpha_1 t_1 + \alpha_2 t_2 + v,  x_{1,1}] = \alpha_1 y_{1,1}$, so that $\alpha_1=0$.
Analogously, the fact that $ \sum_{i=1}^n \beta_i D_i(x_{2,1}) = 0$ now leads to $\alpha_2=0$, so 
$ \sum_{i=1}^n \beta_i D_i = \ad (v)$ for some $v \in K^{4n}$. But then 
\begin{align*}
\sum_{i=1}^n \beta_i y_{2,i} & = \sum_{i=1}^n \beta_i D_i (t_1) = [v, t_1], \\
0 & = \sum_{i=1}^n \beta_i D_i (t_2)=[v,t_2].
\end{align*}

The second equation above shows that $v$ has no components in the $x_{2,i}$'s and thus is $[v,t_1] = 0$. 
Using this in the first equation above leads to $\beta_1=\beta_2=\cdots=\beta_n=0$. \\[0.2cm]
Next we have to verify that the set is generating. Let $D\in \AID(\Lg_n)$. We have to show that 
\[ 
D =  \sum_{i=1}^n \beta_i D_i + \ad (x)
\]
for some $\beta_1,\beta_2,\ldots, \beta_n\in K\mbox{ and } x \in \Lg_n$. Let $D$ be determined by a map $\varphi_D$.
Many of the basis vectors turn out to be fixed: \\[0.3cm]
$1.$ As any vector $y_{p,q}$ belongs to the center of $\Lg_n$, all of these basis vectors are fixed. \\[0.1cm]
$2.$ Also any vector $x_{1,i}$ is fixed, since its centralizer is of codimension 1 in $\Lg_n$, see Lemma \ref{3.3}. \\[0.1cm]
$3.$ To see that $t_2$ is fixed, note that the basis vectors not belonging to $C_{\Lg_n}(t_2)$ are the vectors
$x_{2,i}$. When we apply Lemma \ref{3.7} with $e_i=t_2,\;e_j=x_{2,i},\;e_k=x_{2,j},\;e_l=y_{1,i}$ and $e_m=y_{1,j}$ we can deduce that 
\[ t_{t_2}(\varphi_D(x_{2,i})) = t_{t_2}(\varphi_D(x_{2,j})) \]
from which is follows that $t_2$ is fixed. \\[0.1cm]
$4.$ To see that  $t_1$ is fixed, we start with applying Lemma \ref{3.7} with $e_i=t_1$, $e_j=x_{1,i}$, $e_k=x_{1,j}$, 
$e_l=y_{1,i}$ and $e_{m}= y_{1,j}$. This gives us that 
\[ t_{t_1}(\varphi_D(x_{1,i}) ) =  t_{t_1}(\varphi_D(x_{1,j}) )\mbox{ for all } 1\leq i,j\leq n.\]
Now applying Lemma \ref{3.7} for $i\neq j$ with  
 with $e_i=t_1$, $e_j=x_{1,i}$, $e_k=x_{2,j}$, $e_l=y_{1,i}$ and $e_{m}= y_{2,j}$ we find that 
 \[ t_{t_1}(\varphi_D(x_{1,i}) ) =  t_{t_1}(\varphi_D(x_{2,j}) )\mbox{ for all } 1\leq i,j\leq n \mbox{ with }i\neq j.\]
 Together with the above (and knowing that $n\geq 2$) we can conclude that 
 \[ t_{t_1}(\varphi_D(x_{1,i}) ) =  t_{t_1}(\varphi_D(x_{1,j}) ) =  t_{t_1}(\varphi_D(x_{2,k}) )\mbox{ for all } 1\leq i,j,k\leq n,\]
 showing that $t_1$ is fixed. \\[0.3cm]
The only basisvectors which are not fixed are the vectors $x_{2,i}$. By applying Lemma \ref{3.5} for every fixed basis 
vector, we may now assume that, after changing $D$ up to an inner derivation, we have for each basis vector $x$ that 
$\varphi_D(x)=\sum_{i=1}^n \beta_i(x) x_{2,i}$ for some $\beta_i(x) \in K$.
By changing $D$ to $D+\ad (\varphi_D(t_2))$, we may suppose that $\varphi_D(t_2) = 0$,  and so $D(t_2) = [t_2,\varphi_D(t_2)]=0$. 
Let $x$ be one of the basis vectors $x_{p,i}$ or $y_{p,i}$, then also 
\[ D(x) = [x, \varphi_D(x)] = [ x, \sum_{i=1}^n \beta_i(x) x_{2,i}] =0.\]
Finally, we have that 
\[ D(t_1) = [t_1, \varphi_D(t_1)] = [ t_1, \sum_{i=1}^n \beta_i(t_1) x_{2,i}]= \sum_{i=1}^n \beta_i(t_1)y_{2,i}. \]
As a conclusion, we find that, after changing $D$ up to an inner derivation, we obtain 
\[ D =  \sum_{i=1}^n \beta_i(t_1) D_i.\]
\end{proof}

\begin{rem}
For $n=1$ the basis vector $t_1$ is not fixed. Then the algebra $\Lg_1$ of the above family
is isomorphic to $\Lg_{6,1}$ of Remark $\ref{8.6}$, which is also the algebra of 
Example $(i)$ of  \cite{GOW}, page $245$. For this algebra we know that
$\dim(\AID(\Lg_1)/\Inn(\Lg_1)) = 2$. 
\end{rem}

\section{Acknowledgements}
The first author was supported in part by the Austrian Science Fund (FWF), grant P28079 and grant I3248. 
The second and third author are supported by  long term structural funding -- Methusalem grant of the Flemish Government.

\end{document}